\documentclass{conm-p-l}

\newtheorem{theorem}{Theorem}[section]
\newtheorem{lemma}[theorem]{Lemma}
\newtheorem{proposition}[theorem]{Proposition}
\newtheorem{corollary}[theorem]{Corollary}

\theoremstyle{definition}

\theoremstyle{remark}
\newtheorem{remark}[theorem]{Remark}
\newtheorem{question}[theorem]{Question}
\newtheorem{problem}[theorem]{Problem}
\numberwithin{equation}{section}

   \usepackage{amsmath,amssymb,amsthm}
     \usepackage{graphicx}
  \usepackage{enumerate}
  \usepackage{tikz}
  \input xy
\xyoption{all}

\input xy
\xyoption{all}

 \def\C{\mathbb C}
 \def\P{\mathbb P}
 
\def\Mod{\mathrm{Mod}}

\def\sC{\mathcal C}
\def\sW{\mathcal W}

\def\sB{\mathcal B}
\def \sA{\mathcal A}
\def\sD{\mathcal D}
\def\sM{\mathcal M}
\def\sT{\mathcal T}
\def\Moduli{\sM}
\def\Teich{\sT}

\def\sV{\mathcal V}

\def\sZ{\mathcal Z}

\def\sL{\mathcal L}

\def\Def{\mathcal{D}}

\def\Per{\mathrm{Per}}
\def\V{\mathcal{V}}
\def\AugV{\mathcal{AV}}

\def\AugTeich{\mathcal{AT}}
\def\AugDef{\mathcal{AD}}
\def\AugModuli{\mathcal{AM}}

\def\Eq{\mathrm{Eq}}

\def\sN{\mathcal N}

\begin{document}
\title{The augmented deformation space of rational maps}

\author{Eriko Hironaka}
\thanks{\small \noindent 
This work was partially supported by a grant from the Simons Foundation \#426722. }
\begin{abstract}  
It was recently shown that the
Epstein deformation space of marked rational maps with prescribed combinatorial and dynamical
structure can be disconnected. For example, 
the family of quadratic rational maps with a periodic critical cycle of order 4 and an extra critical
value  not lying in this cycle has a deformation space with infinitely many components.
We study the structure of the augmented deformation space for this example,
and show, in particular, that the closure of deformation space in augmented deformation space
 is also disconnected in this case. \end{abstract}

\maketitle
\begin{center}
{\em In celebration of L\^e D\~ung Tr\'ang's 70th birthday}
\end{center}

\section{Introduction}

Let $A$ and $B$ be two finite subsets of the 2-dimensional sphere $S^2$ such that $|A|,|B| \ge 3$, and
let $f,\iota : (S^2,A) \rightarrow (S^2,B)$ be two maps of pairs:
$$
f : (S^2, A) \rightarrow (S^2,B)
$$
a branched covering with branch values contained in $B$, and 
$$
\iota : (S^2,A) \hookrightarrow (S^2,B)
$$
a homeomorphism identifying domain and range, and $A$ with a subset of $B$.

The {\it Epstein deformation space} $\Def_{f,\iota}$  is defined as the equalizer of the induced maps on Teichm\"uller spaces
$$
f^*,\iota^* : \Teich_{(S^2,B)} \rightarrow \Teich_{(S^2,A)}.
$$
 In unpublished work from the 1990s, A. Epstein showed that if $f$ is not Latt\`es\footnote{See \cite{Milnor:Lattes} for definitions.}, then  $\Def_{f,\iota}$ is either empty or 
a complex $|B| - |A|$ dimensional  submanifold of  $\Teich_{(S^2,B)}$ \cite{Epstein}. 
This generalizes a seminal result of W. Thurston who showed that in the {\it postcritically finite} case, when $ A = B$, 
either $\Def_{f,\iota}$ is empty, $f$ is a Latt\`es example, or $\Def_{f,\iota}$ contains exactly one point \cite{DH93}.  
In particular, $\Def_{f,\iota}$ is always connected in the postcritically finite case. 
When $f$ is not post-critically finite, then $\Def_{f,\iota}$ need not be connected \cite{HK:rational}, but so far there is only
one known class of examples.

\begin{problem}  What are necessary and sufficient conditions for $\Def_{f,\iota}$ to be connected?
\end{problem}

The counter-example to connectedness in \cite{HK:rational} is part of a class of quadratic rational maps studied by
Milnor in \cite{Milnor:Per}.   These are denoted $\Per_n(0)$ and
have one periodic $n$-cycle 
containing a single critical point, and the behavior of the other critical value under iteration is unspecified.  If in addition the extra critical
value does not lie in the first critical $n$-cycle, the rational map is said to be in $\Per_n(0)^*$.
To an element $F \in \Per_n(0)^*$ let $f,\iota : (S^2,A) \rightarrow (S^2,B)$ be
defined so that $f$ is the topological covering underlying $F$, $\iota$ is a identification of
domain and range of $f$ so that  $A \subset B$,  $A$ is the critical $n$-cycle, and $B$ is the union of $A$ and the extra  critical point.  
For $n=3$, $\Teich_{(S^2,A)}$ is a singleton set, and hence
$\Def_{f,\iota}$ is the entire space $\Teich_{(S^2,B)}$.
For $n=4$, the work in \cite{HK:rational} shows that $\Def_{f,\iota}$ is not connected, and in fact has infinitely many
connected components.  This leads to the following question.

\begin{question} Does $\Def_{f,\iota}$ have a natural connected closure?
\end{question}

The points in $\Def_{f,\iota}$ are sometimes called the {\it dynamical points} in the space of complex structures
on $\Teich_{(S^2,B)}$.  The {\it augmented deformation space} $\AugDef_{f,\iota}$,
or {\it ideal dynamical points}, is the subset  of the Bers
augmented deformation space $\AugTeich_{(S^2,B)}$  (see \cite{Bers:ATS}) defined as the equalizer of extensions
$$
\widetilde f^*,\widetilde\iota^* : \AugTeich_{(S^2,B)} \rightarrow \AugTeich_{(S^2,A)}
$$
of $f^*$ and $\iota^*$.  

At the time of this writing, it is not known whether $\AugDef_{f,\iota}$ is connected in the $\Per_4(0)^*$ case.   Our goal in this paper is to 
extend the techniques of \cite{HK:rational}, and apply them 
to give a partial description of the structure of $\AugDef_{f,\iota}$ and the closure of $\Def_{f,\iota}$
within it.  We prove the following theorem.

\begin{theorem} \label{main-thm}  For $(f,\iota)$ associated to an element  of $\Per_4(0)^*$,
the closure of $\Def_{f,\iota}$ in augmented deformation space $\AugDef_{f,\iota}$ is not connected.
\end{theorem}

\subsection{Some background and ideas behind the proofs}
The question of whether and when $\Def_{f,\iota}$ is connected has roots in work of Thurston from the 1980s, in which he
showed that if 
$$
F : \P^1 \rightarrow \P^1
$$
is a non-Latt\`es rational map from the complex projective line 
to itself with a finite post-critical set
$$
 \bigcup_{n=1}^\infty F^{(n)}(\mbox{Crit}_F),
$$
and $f : (S^2,P) \rightarrow (S^2,P)$ is the corresponding branched covering of pointed spheres with domain and range identified, 
and postcritical set $P$,
then the lifting map on holomorphic markings defines a contracting map on Teichm\"uller space
$$
f^* : \Teich_{(S^2,P)} \rightarrow \Teich_{(S^2,P)}.
$$
Thus, $f^*$ has a unique fixed point, and hence $\Def_{f,\iota}$ is a singleton set.

This classical result suggests that there could be a contracting flow in the general case
when the identification map $\iota: (S^2,A) \rightarrow (S^2,B)$ is a strict inclusion on $A$.  
If such a flow exists,  there are two possibilities: one is that $\Def_{f,\iota}$ is connected, and the other is that some
points may be pushed out to the boundary of $\Teich_{(S^2,B)}$.  This suggests looking at dynamical elements
of the augmented Teichm\"uller space to find a natural connected completion of $\Def_{f,\iota}$.

The proof in \cite{HK:rational}, that $\Def_{f,\iota}$ can be disconnected translates the problem about flows on  Teichm\"uller spaces
to one about the topology of algebraic varieties.   A key ingredient is to define an intermediate covering $\Moduli_f$
of $\Teich_{(S^2,B)} \rightarrow \Moduli_{(S^2,B)}$ that is a natural quotient of the space of marked rational maps
combinatorially equal to $f$.    The projection  $\Moduli_f \rightarrow \Moduli_{(S^2,B)}$ is a finite covering, and 
hence has the structure of a quasi-projective variety.  

 In the case when $F \in \Per_4(0)^*$, the space $\Moduli_f$ is isomorphic to a Zariski
dense subset of $\P^1 \times \P^1$:
$$
\Moduli_f = \P^1 \times \P^1 \setminus \sZ \cup \sL
$$
where $\P^1 \times \P^1 \setminus \sL = \Moduli_{(S^2,A)} \times \Moduli_{(S^2,A)}$, and $\sZ$ indicates the locus
where the both critical points of $f$ lie in the same periodic cycle.
With respect to this parameterization, the  image of $\Def_{f,\iota}$ in $\Moduli_f$ is the diagonal
$$
\V_{f,\iota}= \{(x,x) \in \P^1 \times \P^1 \ | \ (x,x) \not\in \sZ \cup \sL\}.
$$
Let $L_f$ be the group of covering automorphisms of $\Teich_{(S^2,B)}$ over $\Moduli_{f,\iota}$.  It was shown in \cite{HK:rational} that
the projection of $\Def_{f,\iota}$ on $\V_{f,\iota}$ is a regular covering with covering automorphism
group a proper subgroup $S_{f,\iota} \subset L_f$.  In particular, $S_{f,\iota}$ acts transitively on fibers of $\Def_{f,\iota} \rightarrow \V_{f,\iota}$.

Choose a basepoint $d_0 \in \Def_{f,\iota}$ and  let $v_0$ be its image in $\V_{f,\iota}$.  Then 
we can identify $L_f$ with $\pi_1(\Moduli_f,v_0)$ and $S_{f,\iota}$ with a subgroup of the fundamental group.  Let $E_{f,\iota}$ be
the image of $\pi_1(\V_{f,\iota},v_0) \rightarrow \pi_1(\Moduli_f,v_0)$ induced by the inclusion map.  Then
$E_{f,\iota}$ is the subgroup of $S_{f,\iota}$ that fixes the component of $\Def_{f,\iota}$ containing $d_0$;
$E_{f,\iota}$ has infinite index in $S_{f,\iota}$; and hence $\Def_{f,\iota}$ has infinitely many components.
In other words, the image of $\pi_1(\V_{f,\iota},v_0)$ is not sufficiently large in $\pi_1(\Moduli_f,v_0)$. 
One way to try to rectify 
this is to put both $\Moduli_f$ and  $\V_{f,\iota}$ into a larger ambient space.

Let $\AugTeich_{(S^2,B)}$ be the augmented Teichm\"uller space of $(S^2,B)$.  Then   $L_f$ 
extends to an action on $\AugTeich_{(S^2,B)}$ giving a quotient $\AugModuli_f$ and 
 a commutative diagram
$$
\xymatrix{
\Teich_{(S^2,B)}\ar[d]_{/L_f} \ar@{^(->}[r] &\AugTeich_{(S^2,B)}\ar[d]^{/L_f}\\
\Moduli_f \ar@{^(->}[r] &\AugModuli_f.
}
$$
Similarly, $S_{f,\iota}$ acts on $\AugDef_{f,\iota}$ with quotient denoted $\AugV_{f,\iota}$.  The spaces 
$\AugV_{f,\iota}$ and $\AugModuli_f$ have the advantage of being algebraic geometric sets, and can be studied 
via singularity theory.  

We define a connected pure 1-dimensional algebraic subset $X$ of a
blowup $\widetilde {\P^1 \times \P^1}$
of $\P^1 \times \P^1$ and an embedding  $X \rightarrow \AugV_{f,\iota}$ that is surjective except possibly
a finite set of points (Proposition~\ref{connectedness-prop}).   By studying
properties of $X$ and its embedding in $\widetilde {\P^1 \times \P^1}$ we prove Theorem~\ref{main-thm}.

\subsection{Organization}  In Section~\ref{teich-sec}, we give necessary definitions of Teichm\"uller and moduli spaces for
rational maps, and their augmented versions.  In Section~\ref{topology-sec},
we prove some general properties of complements of plane algebraic curves 
and their blowups. 
In Section~\ref{proof-sec}, we apply these ideas to the $\Per_4(0)^*$ case, and prove Theorem~\ref{main-thm}.

\subsection{Further questions}

There are still many open questions along the lines of this investigation. So far, the $\Per_4(0)^*$ example is the only known case when $\Def_{f,\iota}$ is disconnected.  Are there others?
Is $\AugDef_{f,\iota}$ connected in the $\Per_4(0)^*$ case, and is $\AugDef_{f,\iota}$ connected in general?  
The analysis in this paper suggest general approaches to these questions, which we leave for further investigation.

\subsection{Some comments on notation}  This paper grew out of the ideas in \cite{HK:rational}, but some of the notation 
has changed.
Most notably, we refer to the Teichm\"uller space $\sT_f$ of
rational maps combinatorially equivalent to $f$, and its corresponding Moduli space $\sM_f$. 
Thus, Epstein's deformation space $\sD_{f,\iota}$ will be considered as a subspace of $\sT_f$ rather than of the isomorphic
space $\sT_{(S^2,B)}$ (also known as the {\it parameter space}).
The space $\sM_f$ may be more familiarly known as a connected component of a Hurwitz space of rational maps \cite{Ramadas:Hurwitz} and was
denoted by $\sW_f$ in \cite{HK:rational}.   
A smaller change is the use of $\sD_{f,\iota}$ for the deformation space $\sD_f$ to emphasize the dependence on both 
the topological branched covering of pairs $f$ and the identification $\iota$.  Similarly, we changed the notation of $S_f$
to $S_{f,\iota}$.  These ease our transition  from a discussion of deformation space to  augmented deformation space.

\subsection{Acknowledgments} I would like to thank 
X. Buff, S. Koch, C. McMullen, R. Ramadas, and L. D. Tr\'ang  for helpful references and discussions, and 
 the anonymous referee for their careful reading and useful comments.

\section{Background definitions}
In this section, we recall definitions and properties of Teichm\"uller space, moduli space and deformation space 
for rational maps and their augmented versions.  For more details about the general theory see, for example, \cite{DMcompactification}, \cite{Bers:ATS},
\cite{HubKoch}, \cite{Ramadas:Hurwitz} and \cite{Selinger:Aug}.  We also recall definitions of liftables $L_f$ and special liftables
$S_{f,\iota}$  from \cite{HK:rational},  and examine their extensions to the augmented spaces.

\subsection{Teichm\"uller spaces for rational maps}\label{teich-sec}
Let $A$ be a finite subset of 3 or more points of the topological 2-sphere $S^2$.  The
{\it Teichm\"uller space} 
$
\Teich_{(S^2,A)}
$
of {\it holomorphic markings} on $(S^2,A)$ is the collection of orientation preserving homeomorphisms
$$
\phi : (S^2,A) \rightarrow (\P^1,\phi(A))
$$
defined up to post-composition by automorphisms of $\P^1$ (i.e., M\"obius transformations) and pre-composition by self-homeomorphisms of $S^2$ that are isotopic rel $A$ to the identity.

Similarly, given a finite branched covering of pairs $f : (S^2, A) \rightarrow (S^2, B)$, where $\infty > |A|, |B| \ge 3$ and
$B$ contains the branch values (or critical values) of $f$, we define the {\it Teichm\"uller space} $\Teich_{f}$ of {\it holomorphic markings} 
on $(f,A,B)$ as the set of  commutative diagrams
$$
\xymatrix{
(S^2,A) \ar[r]^-{\psi}\ar[d]_f &(\P^1,\psi(A))\ar[d]^F\\
(S^2,B) \ar[r]^-\phi& (\P^1,\phi(B))
}
$$
also denoted by $(\phi,\psi,F)$,
where $F$ is a rational map, $\phi$ and $\psi$ are homeomorphisms.  Two triples $(\phi,\psi,F)$ and $(\phi_1,\psi_1,F_1)$ in $\Teich_f$ are
 {\it equivalent} if there  are homeomorphisms $\alpha:  (S^2,A) \rightarrow (S^2,A)$  isotopic to the
 identity map rel $A$ and $\beta:  (S^2,B) \rightarrow (S^2,B)$ isotopic to the identity rel $B$, and biholomorphic maps $\mu,\nu : \P^1 \rightarrow \P^1$
 so that the diagram
 $$
\xymatrix{
(\P^1,\psi_1(A)) \ar[d]_{F_1}&(S^2,A) \ar[l]_-{\psi_1} \ar[d]_f\ar[r]^-{\alpha} &(S^2,A)\ar[d]_f\ar[r]^\psi &(\P^1,\psi(A))\ar[d]_F\\
(\P^1,\phi_1(B)&(S^2,B) \ar[l] _-{\phi_1} \ar[r]^-{\beta} &(S^2,B) \ar[r]^\phi &(\P^1,\phi(B))
}
$$
commutes, and $ \mu = \psi \circ \alpha \circ \psi_1^{-1}$ and $\nu = \phi \circ \beta \circ \phi_1^{-1}$.

 By these definitions, $\Teich_{f}$ comes with natural surjections
 $$
 \xymatrix{
 &\Teich_{f}\ar[dl]_q\ar[dr]^{p_U} \\
 \Teich_{(S^2,B)} &&\Teich_{(S^2,A)}
}
$$
recording the holomorphic markings of the domain space
($p_U$) and the range space ($q$).  Furthermore, $\Teich_f$ can be thought of as the graph of a {\it lifting map}
$$
f^* : \Teich_{(S^2,B)} \rightarrow \Teich_{(S^2A)}
$$
 defined by $f$.  Thus, $q$ is an isomorphism of holomorphic spaces.

We create an iteration scheme from $f$ by partially identifying the domain $(S^2,A)$ and range $(S^2,B)$ of $f$.  That is, 
we fix a homeomorphism $\iota : S^2 \rightarrow S^2$ that restricts to an inclusion $\iota |_A : A \hookrightarrow B$.  
Then we have a map of pairs
$$
\iota : (S^2,A) \rightarrow (S^2,B).
$$
Let $p_L = \iota^* \circ q$, where $\iota^* : \Teich_{(S^2,B)} \rightarrow \Teich_{(S^2,A)}$ is the map that
takes $\phi \in \Teich_{(S^2,B)}$ to the class in $\Teich_{(S^2,A)}$ defined by 
$$
\phi \circ \iota : (S^2,A) \rightarrow (\P^1,(\phi \circ \iota)(A)) = (\P^1, \phi(A)).
$$
Then the elements of $\Teich_f$ can be thought of as the  holomorphic markings of the branched covering $f$, and $p_U$ and $p_L$ record
the induced marked  holomorphic structures $(S^2,A)$ by the domain and range.  

The 
Epstein's {\it deformation space} $\Def_{f,\iota}$ is 
the subspace of $\Teich_f$ 
consisting of holomorphic structures on the covering and base of $f$ that are equivalent relative to $A$.  That 
is,  $\Def_{f,\iota}$ consists of the triples $(\phi,\psi,F)$ so that
$\phi^{-1} \circ \psi$ is  isotopic to the identity relative to $A$, or equivalently 
$$
f^*\phi = \iota^*\phi.
$$
Another way to say this is that $\Def_{f,\iota}$ is the equalizer in $\Teich_f$ of the maps $p_U$ and $p_L$, that is
$$
\Def_{f,\iota} = \{ \phi \in \Teich_f \ | \ p_U(\phi, i) = p_L(\phi)\}.
$$
We think of $\Def_{f,\iota}$ as the {\it dynamical Teichm\"uller space} for $(f,\iota)$.  

\subsection{Moduli spaces}\label{moduli-sec}
Let $\Moduli_{(S^2,A)}$ be the space of
embeddings
$$
A \hookrightarrow \P^1
$$
up to post-composition by a M\"obius transformation.  Then the restriction map $[\phi] \mapsto [\phi |_A]$
defines a regular covering map
$$
\Teich_{(S^2,A)} \rightarrow \Moduli_{(S^2,A)}
$$
with covering automorphism group equal to the 
 {\it mapping class group}  $\Mod(S^2,A)$
 of orientation preserving homeomorphisms $h : S^2 \rightarrow S^2$ that
 fix the points of $A$ up to isotopy rel $A$.  (Unlike the case for Teichm\"uller spaces and moduli space of
 general surfaces, $\Mod(S^2,A)$ on $\Teich_{(S^2,A)}$ acts without fixed points.)

Similarly let $\Moduli_f$ be the space of commutative diagrams
 $$
 \xymatrix{
 A\ar[d]_{f |_A}\ \ar@{^(->}^j [r] &\P^1\ar[d]^F\\
 B\ \ar@{^(->}^i [r]&\P^1.
 }
 $$
 where  $(i,j,F)$ is defined up to modifications of the domain and range of $F$ by M\"obius transformations.
 Then the map $(\phi,\psi, F) \mapsto (\phi |_B, \psi |_A, F)$ defines a covering map
 $$
 \Teich_f \rightarrow \Moduli_f,
 $$
with covering automorphism group $L_f \subset \Mod(S^2,B)$ called the subgroup of {\it liftables}
consisting of elements $h \in \Mod(S^2,B)$ such that for some $h' \in \Mod(S^2,A)$ the diagram
$$
\xymatrix{
(S^2,A)\ar[d]_-f \ar[r]^{h'} &(S^2,A)\ar[d]^-f \\
(S^2,B) \ar[r]^h &(S^2,B)
}
$$
commutes.

\begin{remark}
Since $A$ and $B$ are assumed to contain at least three points, $f: (S^2,A) \rightarrow (S^2,B)$ can have no non-trivial covering automorphisms.
In the degree 2 case, this follows from the fact that $f$ must have exactly two branch points and two branch values.  Any other marked point would
lie in the unbranched part of the covering.   Even in the case when the degree of $f$ is greater than 2, the fact that $f$ may be realized as a 
rational map implies that all covering automorphisms must be conjugate to a M\"obius transformation.  If there are at least three points in $A$,
then the M\"obius transformation must be the identity.
\end{remark}

By the definition of $L_f$ and since $f$ has no non-trivial covering automorphisms,  $f$ defines a 
unique lifting map
$$
f^\sharp: L_f \rightarrow \Mod(S^2,A)
$$
where $f^\sharp h = h'$.

 The following Proposition was shown using slightly different language in \cite{K13} (cf. \cite{HK:rational} Proposition 2.2).
For the convenience of the reader, we include a proof below.

\begin{proposition} The map $\Teich_{f} \rightarrow \Moduli_{f}$ is a regular covering with automorphism group $L_{f}$.
\end{proposition}

\begin{proof}
Let $h \in L_{f}$.   Take any $(\phi,\psi,F) \in \Teich_{f}$.  Then we have
the commutative diagram
\begin{eqnarray}\label{modaction-eq}
\xymatrix{
(S^2,A)\ar[d]_f &(S^2,A) \ar[l]^{f^\sharp h} \ar[r]^{\psi}\ar[d]_f&(\P^1, \psi(A))\ar[d]^F\\
(S^2,B)  & (S^2,B)\ar[l]^h \ar[r]^\phi &(\P^1,\sB).
}
\end{eqnarray}
Since $h$ and $f^\sharp h$ fix $B$ and $A$, respectively,
 $(\phi _B, \psi _A, F)$ and $(\phi\circ h, \psi\circ f^\sharp h, F)$ map to the same element in $\Moduli_{f}$.

Conversely, suppose $(\phi,\psi,F)$ and $(\phi_1,\psi_1, F_1)$ both map to equivalent elements in $\Moduli_{f}$.
Then by definition, there are embeddings $i : B \rightarrow \P^1$ and $j : A \rightarrow \P^1$ so
that $i = \phi |_B = \phi_1 |_B$, $j = \psi_1 |_A = \psi |_A$ and there are M\"obius transformations $\mu$ and $\nu$ such that
$F = \mu^{-1} \circ F_1 \circ \nu$.  We can assume for simplicity that $F = F_1$ by replacing $(\phi_1,\psi_1,F_1)$ by 
the equivalent element $(\mu^{-1} \circ\phi_1,\nu^{-1}\circ \psi_1 , F)$ in $\Teich_{f}$.
Then the common image of $(\phi,\psi,F)$ and $(\phi_1,\psi_1,F)$ is some $(i,j,F)$ in $\Moduli_{f}$.
Thus the images of $\phi$ and $\phi_1$ are the same in $\Moduli_{(S^2,B)}$,  the images of $\psi$ and $\psi_1$
are the same in $\Moduli_{(S^2,A)}$, and hence there are mapping classes $h \in \Mod(S^2,B)$ and $f^\sharp h \in \Mod(S^2,A)$
such that $\phi = h \circ \phi_1$ and $\psi = f^\sharp h \circ \psi_1$.  The maps $h$ and $f^\sharp h$ complete a diagram of the form (\ref{modaction-eq}),
and thus $h \in L_{f}$.
\end{proof}

Let $\overline q : \Moduli_f \rightarrow \Moduli_{(S^2,B)}$ be the map sending $(i,j,F)$ to the equivalence class containing $i$ in $\Moduli_{(S^2,B)}$;
let $\overline p_U: \Moduli_f \rightarrow \Moduli_{(S^2,A)}$ be the map sending $(i,j,F)$ to the equivalence class containing $j$ in $\Moduli_{(S^2,A)}$;
let $\overline \iota^* : \Moduli_{(S^2,B)} \rightarrow \Moduli_{(S^2,A)}$ be the forgetful map sending an inclusion $i : B \rightarrow \P^1$
to $i \circ \iota |_A$; and
let $\overline p_L = \overline \iota^* \circ \overline q$.
Then we have the commutative diagram
$$
\xymatrix{
&\Teich_{f} \ar[dl]_{\simeq}^q \ar[dr]^{p_U}\ar[dd]^{\rho}\\
 \Teich_{(S^2,B)}\ar[dd]_{/ \Mod(S^2,B)}\ar[dr]^{/L_f}&&\Teich_{(S^2,A)}\ar[dd]^{/ \Mod(S^2,A)}\\
&\Moduli_{f} \ar[dl]_{\overline q}\ar[dr]^{\overline p_U}\\
\Moduli_{(S^2,B)} && \Moduli_{(S^2,A)}.
}
$$
All vertical arrows and the left three diagonal arrows are unbranched covering maps.  The right diagonal arrows may not be surjective (see \cite{BEKP:Pullback}).
Let $\V_{f,\iota}$ be the image of the deformation space $\Def_{f,\iota}$ in $\sM_f$.  Then we have
$$
\V_{f,\iota} = \Eq(\overline p_L, \overline p_U).
$$

\subsection{Stabilizer of deformation space}  
Fix a basepoint $d_0 \in \Def_{f,\iota}$, and let $m_0 \in \Moduli_{f,\iota}$ be the image of $d_0$ under the map
$\rho$. Then we have an identification
$$
\ell: \pi_1(\Moduli_{f,\iota},m_0) \rightarrow L_f,
$$
defined by the path-lifting theorem for coverings.  That is, for each $\gamma \in \pi_1(\Moduli_{f,\iota},m_0)$, we lift $\gamma$
to a path $\gamma'$ on $\Teich_{(S^2,B)}$ based at $d_0$, and let $\ell(\gamma)$ be the 
(unique, since the covering is regular) mapping class that takes $d_0$ to the end point 
of $\gamma'$.   

\begin{proposition}  The restriction of $\rho : \sT_f \rightarrow \sM_f$ to 
$\Def_{f,\iota}$ gives a covering map
$$
\rho_{\mathcal D} : \Def_{f,\iota} \rightarrow \V_{f,\iota},
$$
and the image $E_{f,\iota,d_0}$ of 
$$
\pi_1(\V_{f,\iota}, m_0) \rightarrow \pi_1(\Moduli_{f,\iota},m_0) \overset \ell\rightarrow L_f
$$
is the stabilizer of the connected component of $\Def_{f,\iota}$ that contains $d_0$.
\end{proposition}

\begin{proof}  We show that the projection $\rho_{\Def_{f,\iota}}$ satisfies the path-lifting theorem.
Let $d_0 = (\phi_0,\psi_0,F_0)$, and let
$(i_t,j_t,F_t)$ be a path in $\V_{f,\iota}$ with $m_0 = (j_0,i_0,F_0)$.   Let $\xi$ and $\eta$ be a representatives of the class
of $\phi_0$ and $\psi_0$ so that the diagram commutes
$$
\xymatrix{
S^2 \ar[d]_f \ar[r]^\eta &\P^1 \ar[d]^F\\
S^2 \ar[r]^\xi &\P^1,
}
$$
$\eta |_A = \psi_0 |_A$ and $\xi |_B = \psi |_B$.
Let $p_t : (S^2,B) \rightarrow (S^2,B_t)$ be any continuous family of homeomorphisms. Then
$$
(\xi \circ p_t |_{B_t}, \eta \circ p_t |_{A_t}, F_t) \sim (i_t, j_t, F_t)
$$
and $(\xi \circ p_t, \eta \circ p_t, F_t)$ is a lift of $(i_t,j_t,F_t)$ and lies in $\Def_{f,\iota}$.
 
Thus, as a restriction of
an unbranched covering map, $\rho_{\mathcal D}$ is itself is a covering map.  
The rest follows from basic covering space theory. 
\end{proof}

The lifting map  $f^\sharp : L_f \rightarrow \Mod(S^2,A)$ is uniquely defined
and satisfies the commutative diagram
$$
\xymatrix{
(S^2,A)\ar[d]_f \ar[r]^{f^{\sharp}h} & (S^2,A)\ar[d]^f\\
(S^2,B) \ar[r]^h &(S^2,B).
}
$$
For a homeomorphism $h : (S^2,B) \rightarrow (S^2,B)$, let $h_A$ be the element of $\Mod(S^2,A)$ defined by
ignoring the points of $B \setminus A$.  Then $h_A$ is the isotopy class of $h$ defined up to homeomorphisms 
isotopic to the identity rel $A$.  Define
\begin{eqnarray*}
\iota^\sharp : \Mod(S^2,B) &\rightarrow& \Mod(S^2,A)\\
h&\mapsto& h_A.
\end{eqnarray*}

Let $S_{f,\iota} \subset L_f$ be the equalizer
$$
S_{f,\iota} = \Eq(f^\sharp,\iota^\sharp) \subset L_f.
$$
By the identification of $L_f$ with $\pi_1(\Moduli_{f,\iota},m_0)$,  $S_{f,\iota}$ can equivalently be defined as the
equalizer of the  homomorphisms
$$
(\overline p_L)_*, (\overline p_U)_* : \pi_1(\Moduli_{f,\iota},m_0) \rightarrow \pi_1(\Moduli_{(S^2,A)},a_0)
$$
where $a_0 = \overline p_L(m_0) = \overline p_U(m_0)$, where $m_0 \in \V_{f,\iota}  = \Eq (\overline p_L,\overline p_U)$.

\begin{proposition} [\cite{HK:rational} Proposition 2.5] The stabilizer in $L_f$ of $\Def_{f,\iota}$ equals $S_{f,\iota}$, and 
$S_{f,\iota}$ acts transitively on the fibers of the covering
$$
\Def_{f,\iota} \rightarrow \V_{f,\iota}.
$$
Thus the covering is regular and $S_{f,\iota}$ is the group of covering automorphisms.
\end{proposition}

\begin{corollary} \label{Defconsuff-cor} The deformation space $\Def_{f,\iota}$ is connected if and only if $\V_{f,\iota}$ is connected and  $S_{f,\iota} = E_{f,\iota,d_0}$.
\end{corollary}

In the case when $(f,\iota)$ is associated to an element of $\Per_4(0)^*$, $E_{f,\iota,d_0}$ has infinite index in $S_{f,\iota}$ (\cite{HK:rational}, Proposition 2.11).

\subsection{Augmented spaces}\label{augmented-sec}  
By a {\it rational curve} $\sC$ we mean a nodal curve with the following properties
\begin{enumerate}[\hspace{\parindent}(a)]
\item the irreducible components of $\sC$ are isomorphic to $\P^1$, and
\item the fundamental group of $\sC$ is trivial.
\end{enumerate}
A {\it pre-stable rational curve} $(\sC,\sA)$ is a rational curve $\sC$ together with a finite set $\sA$ contained in the complement of the nodes of $\sC$.
The set of nodes of $\sC$ union the points of $\sA$ form 
 the {\it distinguished points} of $\sC$.
A {\it stable rational curve} is a pre-stable rational curve with the following additional property:
\begin{enumerate}[\hspace{\parindent}(c)]
\item the number of distinguished points on each irreducible component of $\sC$ is greater than or equal to $3$.
\end{enumerate}

For each component $C$ of a pre-stable rational curve $(\sC,\sA)$ there are three possibilities:
\begin{enumerate}
\item $C$ is stable, i.e., it contains at least 3 distinguished points;
\item $C$ is unstable, and contains two nodes; or
\item $C$ is unstable, and contains one node and zero or one point in $\sA$.
\end{enumerate}
Let $\Sigma^{\mbox{pre}}_{(S^2,A)}$ be the set of pre-stable rational curves $(\sC,\sA)$ with a bijection
$A \rightarrow \sA$,
and let $\Sigma_{(S^2,A)} \subset\Sigma^{\mbox{pre}}_{(S^2,A)}$ be the space of stable rational curves.
Define a map
$$
\mathfrak s : \Sigma^{\mbox{pre}}_{(S^2,A)} \rightarrow \Sigma_{(S^2,A)}
$$
sending $(\sC,\sA)$ to the result of contracting components of $\sC$ using the following
rule:
in case (1) leave $C$ alone; and in case (2)  and (3) contract $C$ to a point.  If $C$ contains a point of $\sA$, then mark the 
image of the contraction by that point.
This map is well-defined since in case (2) there is at most one point of $\sA$ in $C$.  We call $\mathfrak s$ the {\it stabilization map}.

A {\it marking} of a pre-stable rational curve $(\sC,\sA) \in \Sigma^{\mbox{pre}}_{(S^2,A)}$ is a quotient map
$$
\phi : (S^2,A) \rightarrow (\sC,\sA),
$$
such that $\phi$ is a homeomorphism when restricted to $S^2 \setminus \gamma$ for some multi-curve $\gamma \subset S^2 \setminus A$,
$\phi$ restricts to a  bijection $A \rightarrow \sA$, and the components
of $\gamma$ are in one-to-one correspondence with the nodes of $\sC$.    The curve $\gamma$ is called the {\it contracting curve} for $\phi$.
We consider two markings {\it equivalent} if they are the same up to post-composition by automorphisms 
of $\sC$ and pre-composition by homeomorphisms $(S^2,A) \rightarrow (S^2,A)$ that are isotopic to the identity rel $A$.

The collection of markings of pre-stable and stable rational curves by $(S^2,A)$ is denoted by $\AugTeich^{\mbox{pre}}_{(S^2,A)}$ and
 $\AugTeich_{(S^2,A)}$, respectively.
 Post-composition by $\mathfrak s$ defines a surjection 
 $$
\AugTeich^{\mbox{pre}}_{(S^2,A)} \rightarrow \AugTeich_{(S^2,A)}.
$$
The space $\AugTeich_{(S^2,A)}$ is called the {\it augmented Teichm\"uller space} of $(S^2,A)$.
There is a natural topology on $\AugTeich_{(S^2,A)}$ such that points on
$\AugTeich_{(S^2,A)} \setminus \Teich_{(S^2,A)}$ are the limits of sequences points on $\Teich_{(S^2,A)}$ for which the length of  $\gamma$ tends to zero (see \cite{Bers:ATS} for more precise definitions).

\begin{remark}\label{AugTeichAuts-rem}
The mapping class group $\Mod(S^2,A)$ extends to actions on $\AugTeich^{\mbox{pre}}_{(S^2,A)}$ and $\AugTeich_{(S^2,A)}$.    For points in $\AugTeich^{\mbox{pre}}$ the stabilizer contains a copy of the group of
M\"obius transformations.
Given a point in $\AugTeich_{(S^2,A)}$, the stabilitzer is the free abelian group of mapping classes generated by Dehn twists along
the components of the contracting curve.  
\end{remark}

The action of $\Mod(S^2,A)$ on $\AugTeich_{(S^2,A)}$ defines a branched covering
$$
\AugTeich_{(S^2,A)} \rightarrow \AugModuli_{(S^2,A)}
$$
where $\AugModuli_{(S^2,A)}$ is the space of inclusions
$$
A \hookrightarrow \sC
$$
of $A$ into a rational curve $\sC$
up to holomorphic automorphism of $\sC$ that do not permute components (cf. \cite{DMcompactification}).

We now define  the augmented Teichm\"uller and moduli spaces of $f$.
A rational map $F : (\sC_U,\sA) \rightarrow (\sC_L,\sB)$ is {\it pre-admissible} if
\begin{enumerate}[\hspace{\parindent}(a)]
\item $F$ defines a surjective map from $\sC_U$ to $\sC_L$ of generically constant degree that maps nodes to nodes;
\item locally near each node of $\sC_U$, $F$ has generically constant degree; and
\item $(\sC_L,\sB)$ is a stable rational curve and  $(\sC_U,\sA)$ is pre-stable.
\end{enumerate}
We say that $F$ is {\it admissible} if in addition 
\begin{enumerate}[\hspace{\parindent} (b)]
\item $(\sC_U,\sA)$ is stable.
\end{enumerate}
The {\it augmented Teichm\"uller space} $\AugTeich_f$ for $f$  is the collection of holomorphic markings
\begin{eqnarray*}
\xymatrix {
(S^2,A) \ar[r]^\psi\ar[d]_f &(\sC_U,\sA)\ar[d]^F\\
(S^2,B) \ar[r]^\phi & (\sC_L,\sB)
}
\end{eqnarray*}
where the horizontal maps are markings in $\AugTeich_{(S^2,A)}$ and $\AugTeich_{(S^2,B)}$ respectively, and $F$ is a pre-admissible covering.
Here, as in the definition of $\Teich_f$, we take $(\phi,\psi,F)$ up to the natural equivalences.

With this definition, the projection $\widetilde q : \AugTeich_f \rightarrow \AugTeich_{(S^2,B)}$ defines an isomorphism.
Let $\widetilde p_U, \widetilde p_L : \AugTeich_f \rightarrow \AugTeich_{(S^2,A)}$ be defined by
\begin{eqnarray*}
\widetilde p_L (\phi,\psi,F) &=& \mathfrak s(\phi)\\
\widetilde p_U(\phi,\psi,F)  &=& \mathfrak s(\psi)
\end{eqnarray*}

\begin{proposition} The subgroup of liftables $L_f \subset \Mod(S^2,B)$ extends to an action on $\AugTeich_{f}$.
\end{proposition}

\begin{proof} Let $\phi : (S^2,B) \rightarrow (\sC_L,\sB)$ be an element of $\AugTeich_{(S^2,B)}$, and let $\gamma$ be the contracting multi-curve.
Since $f$ is unbranched outside of $B$, it follows that  $f^{-1}(\gamma)$ is a multi-curve on $S^2 \setminus A$, and
since $h$ is liftable
$$
f^\sharp h(f^{-1}(\gamma)) = f^{-1}(h(\gamma)).
$$
Thus $h \in L_f$ takes $(\phi,\psi,F)$ to $(\phi \circ h, \psi \circ f^\sharp h, F)$ where $\phi \circ h$ and $ \psi \circ f^\sharp h$
contract the curves $h(\gamma)$ and $f^{-1}(h(\gamma))$ respectively.
\end{proof}

Let $\AugModuli_f = \AugTeich_f/L_f$.  Then the points of $\AugModuli_f$ are defined by diagrams
$$
\xymatrix{
A\ar[d]_{f|_A}\ar@{^(->}[r]^i&\sC_U\ar[d]^F\\
B\ar@{^(->}[r]^j &\sC_L
}
$$
where $(\sC_L, \sB)$ is stable, $(\sC_U,\sA)$ is pre-stable, and  the inclusions $i,j$ 
are defined up to holomorphic automorphisms of $\sC_U$ and $\sC_L$.  
We denote an element by $(i,j,F)$.
Let $\overline p_U, \overline p_L : \AugModuli_f \rightarrow \AugModuli_{(S^2,A)}$ be defined by
\begin{eqnarray*}
\overline p_L ([\phi,\psi,F]) &=& \mathfrak s\circ j\\
\overline p_U([\phi,\psi,F])  &=& \mathfrak s\circ i
\end{eqnarray*}
While the action of $L_f$ on $\Teich_f$ has no fixed points and the quotient map $\Teich_f \rightarrow \Moduli_f$ is a covering,
the quotient map $\AugTeich_f \rightarrow \AugModuli_f$ can have branch points.

Summarizing, we have a commutative diagram of augmented spaces:
$$
\xymatrix{
&\AugTeich_{f} \ar[dl]_{\widetilde q}\ar[dr]^{\widetilde p_U}\ar[dd]^{\widetilde\rho}\\
\AugTeich_{(S^2,B)}\ar[dd]_{/ \Mod(S^2,B)}\ar[dr]^{/L_f}&&\AugTeich_{(S^2,A)}\ar[dd]^{/ \Mod(S^2,A)}\\
&\AugModuli_{f} \ar[dl]_{\overline q}\ar[dr]^{\overline p_U}\\
\AugModuli_{(S^2,B)} && \AugModuli_{(S^2,A)},
}
$$
The augmented deformation space is defined to be the equalizer
$$
\AugDef_{f,\iota} = \Eq(\widetilde p_U,\widetilde p_L)
$$
and contains the deformation space $\Def_{f,\iota}$.  Let $\AugV_{f,\iota} = \widetilde \rho(\AugDef_{f,\iota})$.
Then we have
$$
\AugV_{f,\iota} = \Eq(\overline p_U,\overline p_L),
$$

\subsection{Stabilizer of augmented deformation space}\label{stab-sec}
In this section, we give a necessary and sufficient condition for $\AugDef_f$ to be connected.  Unlike in the case
for Corollary~\ref{Defconsuff-cor}, $\AugDef_f \rightarrow \V_f$ is not an unbranched covering.  To get around this
we use the notion of regular neighborhoods.
 Recall that, for a simplicial subcomplex $V$ embedded in a manifold $X$, 
a {\it regular neighborhood} $N(V)$ of $V$ is an open subset of $X$ containing $V$ that 
has a deformation retract to $V$. 
 
\begin{proposition} The stabilizer in $L_f$ of $\AugDef_{f,\iota}$ equals $S_{f,\iota}$, that is, if $g \in L_f$ is such that
$g(\alpha) = \alpha$ for all $\alpha \in \AugDef_{f,\iota}$, then $g \in S_{f,\iota}$.
\end{proposition}

\begin{proof} The stabilizer in $L_f$ of $\AugDef_{f,\iota}$ must be contained in $S_{f,\iota}$, since $\Def_{f,\iota} \subset \AugDef_{f,\iota}$.
Let $h \in S_{f,\iota}$ and let $(\phi,\psi, F) \in \AugDef_{f,\iota}$.  Then by definition $f^\sharp h = i^\sharp h$.  Thus,
we have a commutative diagram
$$
\xymatrix{
(S^2,A) \ar[r]^{f^\sharp h}\ar[d]_f& (S^2,A) \ar[d]^f\ar[r]^\psi& (\sC_U,\sA)\ar[d]^F\\
(S^2,B) \ar[r]^h &(S^2,B) \ar[r]^\phi & (\sC_L,\sB)
}
$$
and we have
$$
\widetilde f^* (h([\phi]))  = [\psi \circ f^\sharp h] =  [\psi \circ\iota^\sharp h] =\widetilde \iota^*(h([\phi])).
$$
Thus $h$ stabilizes $\AugDef_{f,\iota}$.  
\end{proof}

\begin{remark}  Conversely, one can ask whether if $g \in L_f$ and $g(\alpha) = \alpha$ for some $\alpha \in \AugDef_{f,\iota}$, then does it follow that
$g \in S_{f,\iota}$?  This is not true in general, since points in the boundary of $\AugDef_{f,\iota}$ have extra automorphisms that need not be
in $S_{f,\iota}$ (see Remark~\ref{AugTeichAuts-rem}).
\end{remark}

\begin{proposition} \label{suff-prop} Suppose there is a  connected quasi-projective variety
$X$ with $\V_{f,\iota} \subset X \subset \AugV_{f,\iota}$ such that
$X$ has a regular neighborhood $N(X) \subset \AugModuli_f$ with the properties
\begin{enumerate}
\item $N(X) \cap \Moduli_{f}$ is connected,
and 
\item the image
of the homomorphism induced by inclusion
$$
\pi_1(N(X) \cap \Moduli_f,m_0) \rightarrow \pi_1(\Moduli_f,m_0) 
$$
contains $S_{f,\iota}$.  
\end{enumerate}
Then $\Def_{f,\iota}$ is contained in a connected component of $\AugDef_{f,\iota}$.
\end{proposition}

\begin{proof}  Let $Y_0 \subset \AugTeich_f$ be the connected component of the preimage of $X$ in $\AugTeich_f$ containing $d_0$.
Let 
$U$ be the connected component of the preimage of $N(X)$ in $\AugTeich_f$ containing $d_0$.  Then 
since $N(X)$ has a retract to $X$, the set $U$ has a corresponding
retraction to a component of $\widetilde \rho^{-1}(X)$.  This component is necessarily $Y_0$, since
$Y_0 \cap U \neq \emptyset$. 
Since the image of $\pi_1(N(X) \cap \Moduli_f,m_0)$ in $L_f = \pi_1(\Moduli_f,m_0)$ contains $S_{f,\iota}$, it follows that the action of $S_{f,\iota}$ on  $\AugTeich_f$ preserves $U$ and hence $Y_0$.
Thus, we have
$$
\Def_{f,\iota} \subset Y_0 \subset \AugDef_{f,\iota}.
$$
\end{proof}

\section{Blowups and topology of curve complements}\label{topology-sec}

In this section we study the topology of surface/curve pairs and the effect of blowups  (see for example \cite{Fulton:AC} or \cite{Hartshorne}
for a review of elementary blowup theory for surfaces).

\subsection{Regular neighborhoods of algebraic curves on surfaces}\label{reg-sec}

Let $X$ be a smooth complex projective surface, and let $V \subset X$ be a pure codimensional one subvariety.
We first observe that we may assume that $V$ is a finite union of smooth curves with normal crossings.
 Let $Q$ be the set of singular points on $V$.  Then $V \setminus Q$ is a finite union of smoothly embedded
 punctured Riemann surfaces in $X \setminus Q$.
 By successively blowing up $X$ at the points of $Q$ (and at points of the preimages of $Q$),
it is possible to obtain a new projective surface $\widetilde X$ and a surjective morphism $\sigma: \widetilde X \rightarrow X$ such that
the preimage (or {\it total transform}) $\widetilde V = \sigma^{-1}(V)$ is a union of smoooth curves with normal crossings.    
That is $(\widetilde X, \widetilde V)$
is locally isomorphic near a point of intersection on $\widetilde V$ to $(\C^2, \{x=0\} \cup \{y=0\})$.

Hereafter in this section we assume that $V$ has smooth components intersecting in normal crossing.
As before, let $Q$ be the set of intersections of $V$, and let $V_1,\dots, V_k$ be the irreducible components of $V$.
Since each $V_i$ is smooth there is an embedded tubular neighborhood $T(V_i) \subset X$ so that
for $i\neq j$, $T(V_i)$ only intersects $T(V_j)$ near points in $Q$ where $V_i$ and $V_j$ intersect,
and if $V_i$ and $V_j$ intersect at $q$, then 
$$
N(q) = T(V_i) \cap T(V_j)
$$
is a neighborhood of $q$ so that $(N(q),V,q)$ is homeomorphic to 
$$
(\{|x|<1\} \times \{|y|<1\}, \{x=0\} \cup \{y=0\}, (0,0)).
$$

For $i=1,\dots,k$, let
$$
V_i^c = V_i \setminus \bigcup_{q \in Q} N(q)
$$
and let $T(V_i^c)$ be the tubular neighborhood given by 
$$
T(V_i^c) = T(V_i)  \setminus \bigcup_{q \in Q} N(q).
$$
Let $T(V) = \bigcup_{i=1}^k T(V_i^c)$, called the {\it tubular neighborhood} of $V$.

Let $S(V_i^c)$ be the circle bundle over $V_i^c$ contained in the boundary of $T(V_i^c)$.  Then $S(V_i^c)$ is an oriented 3-manifold
with torus boundary components corresponding to the intersections of $V_i$ with other components of $V$.  In particular, if $V_i$ is
isomorphic to $\P^1$, then $S(V_i^c)$ is the complement of thickened Hopf links in the 3-sphere.
The {\it boundary manifold} of $V$ is given by
$$
S(V) = \bigcup_{i=1}^k S(V_i^c).
$$
This manifold and its embedding in $X \setminus V$ is uniquely determined up to homeomorphisms of $X$ that are isotopic to the identity rel $V$.

\begin{lemma} \label{contract-lem} The punctured tubular neighborhood $T(V) \setminus V$ has a deformation retraction to $S(V)$.
\end{lemma}

\begin{proof}  For each $i$, $T(V_i^c) \setminus V_i^c$ has a deformation retract to $S(V_i^c)$ corresponding from the retraction of a punctured disk
to its boundary circle.  Thus, we have only to consider what happens near the intersection points $q \in Q$.  In $N(q)$ it is enough to show that
$$
\{ |x| < 1\} \times \{ |y| < 1\}   \setminus \{x=0\} \cup \{y=0\}
$$
has a deformation retract to $\{|x|=1\} \times \{|y|=1\}$.  Such a retraction is defined by the map
$$
((x,y),t) \mapsto \left (\frac{x}{1 + t(|x|-1)}, \frac{y}{1+t(|y|-1)}\right ).
$$
\end{proof}

\begin{remark}\label{graphmanifold-rem}
By its construction, $S(V)$ is naturally homeomorphic to a  graph manifold (see, for example, \cite{Hemp:Res} for definitions) over the incidence graph 
$\Gamma$ of the components of $V$;
this is the bipartite graph with vertices 
$$
\{v_i \ | \ i=1,\dots k\}
$$
corresponding to the components of $V$, and edges between $v_i$ and $v_j$ for each $q \in Q$ such that $V_i$ and $V_j$ intersect
at $q$.
To each vertex $v_i$  associate 
the manifold $S(V_i^c)$ and to each edge of $\Gamma$ between $v_i$ and $v_j$ associate the common torus boundary component of their associated vertex manifolds.  
\end{remark}

\subsection{Regular neighborhoods and fundamental groups}\label{fund-sec}
In this section we prove an easy
variation of the Lefschetz hyperplane theorem \cite{AF:Lefschetz} and a useful corollary.

\begin{lemma}\label{surj-lem}  Let $\mathfrak p : X \rightarrow \P^1$ be a smooth projective surface fibered over the complex projective line,
and let $V$ be a fiber.
Let  $C \subset X$ be a pure codimension one algebraic subset none of whose components are fibers of $\mathfrak p$.
Then $V$ has a regular neighborhood $N(V)$ so that 
$$
\pi_1(N(V) \setminus C) \rightarrow (X \setminus C)
$$
induced by inclusion is surjective.
\end{lemma}

\begin{proof}  Let $P \subset \P^1$ be the set of points $p$ where $\mathfrak p$ restricted to $C$ drops in cardinality, i.e., at
least one intersection point of $C$ and the fiber above $p$ has higher multiplicity. Let $\mathfrak p^o$
be the restriction of $\mathfrak p$ to $X^o = \mathfrak p^{-1}(\P^1 \setminus P) \setminus C$.  Then
$$
\mathfrak p^o: X^o \rightarrow \P^1 \setminus P
$$
is a fiber bundle, and the Zariski-Van Kampen theorem \cite{Kampen:Fund} \cite{Chen:Van} implies that for a general fiber $V'$ of $\mathfrak p^o$
$$
\pi_1(X \setminus C) \simeq \pi_1(V')/K
$$
where $K$ is the subgroup of $\pi_1(V')$ generated by the relations $\gamma^{-1} \beta (\gamma)$, where
$\beta$ ranges over automorphisms of $\pi_1(V')$ determined by the action of $\pi_1(\P^1 \setminus P)$ on
$V'$.   In particular, the homomorphism 
$$
\pi_1(V') \rightarrow \pi_1(X \setminus C)
$$
is surjective.   Let
$N(V) =  {(\mathfrak p^o)}^{-1}(U)$ where $U$ is a small neighborhood of $\mathfrak p (V)$.
Then $N(\widetilde V) \setminus C$ contains a general fiber of $\mathfrak p^o$, and 
the claim follows.
\end{proof}

\begin{corollary}\label{connectedness-cor}  Let $C \subset \C^2$ be an algebraic curve, and let $V \subset \C^2$ be a line not contained in $C$.
Then we can include $\C^2$ as a Zariski open subset of a smooth projective surface $X$, and find a sequence of blowups
$\sigma : \widetilde X \rightarrow X$
such that 
\begin{enumerate}
\item $\sigma$ is an isomorphism over $\C^2 \setminus C$; and
\item the total transform $\widetilde V$ of the closure of $V$ in $X$
has a regular neighborhood $N(\widetilde V)$ in $\widetilde X$ such that
the map on fundamental groups induced by inclusion
$$
\pi_1(N(\widetilde V) \cap \C^2 \setminus C) \rightarrow \pi_1 (\C^2 \setminus C)
$$
is surjective.  
\end{enumerate}
\end{corollary}

\begin{proof} Let $\mathfrak p : \C^2 \rightarrow \C$ be a linear projection so that $V$ is a fiber.   Then
we can define completions $\C^2 \subset \P^1 \times \P^1$ and $\C \subset \P^1$, so that $\mathfrak p$
extends to a projection
$$
\overline {\mathfrak p}: \P^1 \times\P^1 \rightarrow \P^1
$$
so that the closure $\overline V$ of $V$ in $\P^1 \times \P^1$ is a fiber.
Let $\overline  C$ be the union of the closure of $C$ in $\P^1 \times \P^1$, and the two lines in $\P^1 \times \P^1 \setminus \C^2$.

Let $\sigma : \widetilde {\P^1 \times \P^1} \rightarrow \P^1$ be sequence of blowups over points of $\overline C \cap \overline V$
such that the total transform $\widetilde V = \sigma^{-1}(\overline V)$ and the proper transform $\widehat C$ over $\overline C$
meet in normal crossings.
Let $N(\widetilde V)$ be a regular neighborhood of $\widetilde V$ as in Lemma~\ref{surj-lem}.  Then
 the homomorphism 
$$
\pi_1(N(\widetilde V) \setminus \widehat C) \rightarrow \pi_1(\widetilde {\P^1 \times \P^1} \setminus \widehat C)
$$
induced by inclusion is surjective.

Since $\sigma$ is an isomorphism outside the preimage of $\overline C \cap \overline V$,
we have a commutative diagram
 $$
 \xymatrix{
\widetilde V\ar[d] \ar[r] &N(\widetilde V) \ar[d]\\
\overline V \ar[r] &N(\overline V)
}
$$
where the vertical arrows are quotient maps that contract the exceptional curves (or 1-dimensional fibers) over points on $\overline V \cap \overline C$.
It follows that the retraction of $N(\widetilde V) \setminus \widehat C$ to $\widetilde V \setminus \widehat C$ descends to a retraction
of $N(\overline V) \setminus C$ to $\overline V \setminus C$, and hence $N(\overline V) \setminus \overline C$ is a regular neighborhood of $\overline V \setminus \overline C$.
Finally, $P^1 \times \P^1 \setminus \overline C = \C^2 \setminus C$, so setting $N(V) = N(\overline V) \cap \C^2$, it follows 
that $N(V)$ is a regular neighborhood of $V$ and 
the homomorphism
$$
\pi_1(N(V) \setminus C) \rightarrow \pi_1(\C^2 \setminus C)
$$
defined by inclusion is surjective.
\end{proof}

Corollary~\ref{connectedness-cor} is used in our discussion of connectivity of augmented deformation space  in Section~\ref{connectedness-sec}.

\begin{remark}\label{Hopf-rem}
The boundary manifold $S(\widetilde V)$ associated to $N(\widetilde V)$ in the previous proof 
 has the structure of a  boundary manifold over the incidence
graph of the irreducible components of $\widetilde V$ (cf. Remark~\ref{graphmanifold-rem}).  Furthermore, 
each component of $\widetilde V$ is isomorphic to a line,
each vertex manifold is an $S^1$ fiber bundle over $S^2$ with a finite set of thickened fibers removed.
\end{remark}

\section{Application to the Main  Example}\label{proof-sec}

Let $F \in \Per_4(0)^*$, and let 
 $f,\iota : (S^2,A) \rightarrow (S^2,B)$ be the underlying branched covering and identification
 of domain and range so that
$A$ is the periodic 4 cycle, and $B = A \cup \{v\}$ where $v$ is the extra critical point.
We study the inclusion $\Def_{f,\iota} \subset \AugDef_{f,\iota}$ by looking at their images $\V_{f,\iota}$ and $\AugV_{f,\iota}$ in 
$\Moduli_{f}$ and $\AugModuli_{f}$.

\subsection{Parameterization of moduli space}\label{parameterization-sec}
We begin by embedding $\Moduli_f$ in $\P^1 \times \P^1$ as follows.
Consider an element of $\Moduli_f$ represented as a commutative diagram
$$
\xymatrix{
A\ \ar[d]_{f |_A} \ar@{^(->}[r]^i &\P^1\ar[d]^F\\
B\  \ar@{^(->}[r]^j &\P^1.
}
$$
By applying automorphisms of $\P^1$ on the right side, we can assume that
$$
i(B) = \{0,1,\infty, y,z\} \quad j(A) = \{0,1,\infty,x\},
$$
where 
\begin{enumerate}[(i)]
\item $\infty$ and $z$ are the critical values of $f$;
\item   $0$ is a critical point in $A$ with $f(0) = \infty$; and
\item $f(\infty) = 1$, $f(x) = 0$.  
\end{enumerate}

The above data completely determines $F : \P^1 \times \P^1$ as a rational function in the variable $t$:
$$
F(t) = \frac{(t-x)(t-r)}{t^2} \qquad r = \frac{x + y - 1}{x -1}.
$$
It follows that in this example $z$ is determined by $x$ and $y$:
\begin{eqnarray}\label{z-eqn}
z = - \frac{(1 - 2x + x^2 - y)^2}{4 x (x-1)(x+y-1)}.
\end{eqnarray}
We have the following (see also, \cite{HK:rational}).

\begin{lemma} \label{par-lem} There is an identification 
$$
\Moduli_f  =   \P^1 \times \P^1 \setminus \sL \cup \sZ,
$$
assigning $(i,j,F)$ to $(x,y)$,
 where
$$
\sL = \{x = 0\} \cup \{y = 0\} \cup \{x = 1\} \cup \{y = 1\} \cup \{x =  \infty\} \cup \{y = \infty\}
$$
and $\sZ$ is the closure in $\P^1 \times \P^1$ of the affine union of curves
$$
\{1-2x + x^2 - y =0 \}\cup  \{x^2+y = 1\}  \cup \{ x + y = 1\} \cup \{2xy+x^2-y-2x+1=0\}
$$
\end{lemma}

\begin{figure}[htbp] 
   \centering
   \includegraphics[width=1.5in]{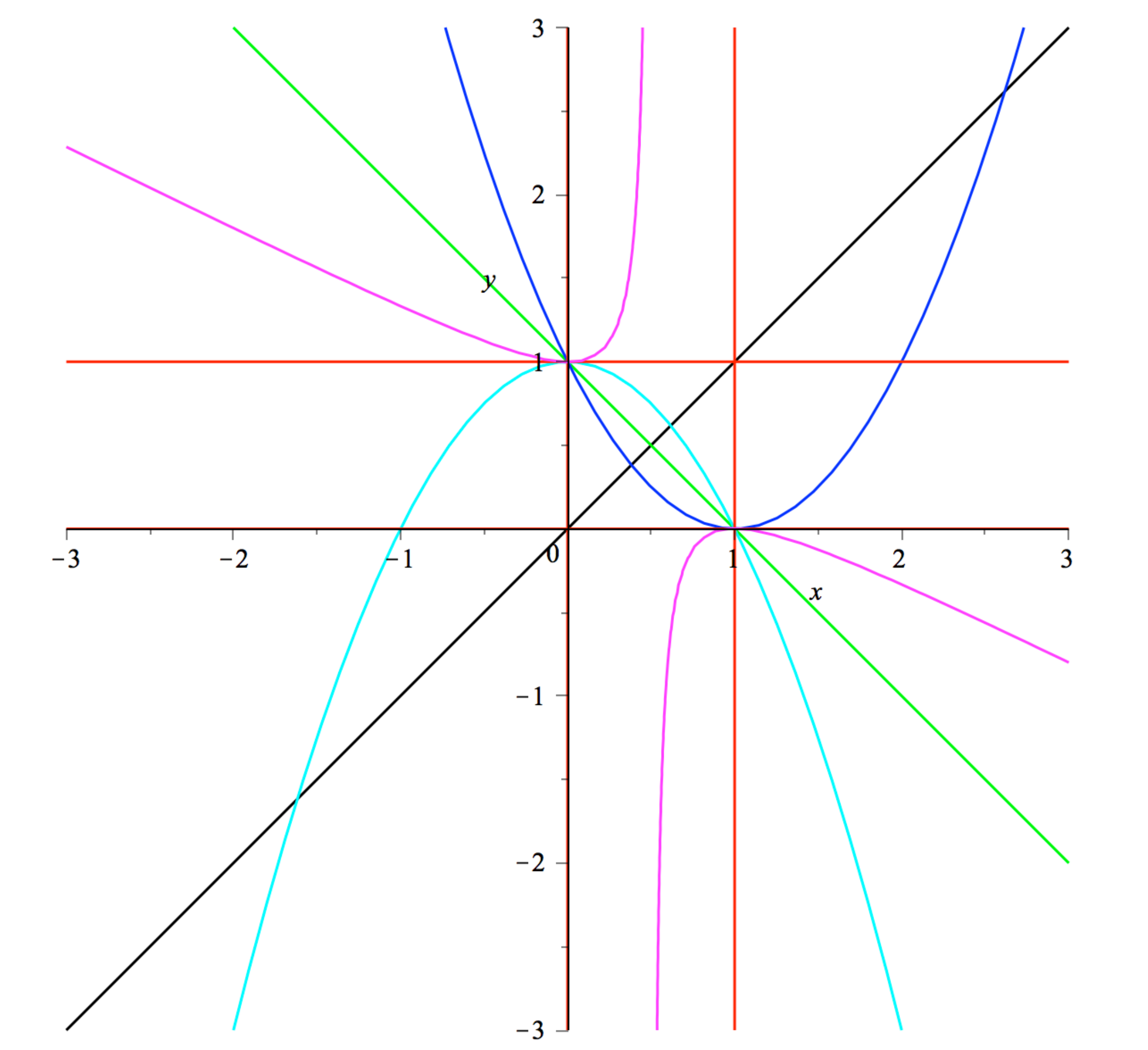} 
   \caption{Picture of $\sL \cup \sZ$ in the affine plane. The lines in $\sL$ are drawn in red.}
   \label{affine-fig}
\end{figure}

By this parameterization, the image $\V$ of $\Def_{f}$ in $\Moduli_{f}$ equals the diagonal
$$
\V = \{(x,y) \in \P^1 \times \P^1 \setminus \sL \cup \sZ \ | \ x = y\}.
$$
Figure~\ref{affine-fig} gives a picture\footnote{This figure was provided courtesy of Sarah Koch.} of the real part
of $\sL \cup \sZ$ in the affine open subset $\C^2 \subset \P^1 \times \P^1$, and Figure~\ref{ClosedW-fig} gives a picture near $(\infty,\infty)$.

\begin{figure}[htbp] 
   \centering
   \includegraphics[width=1in]{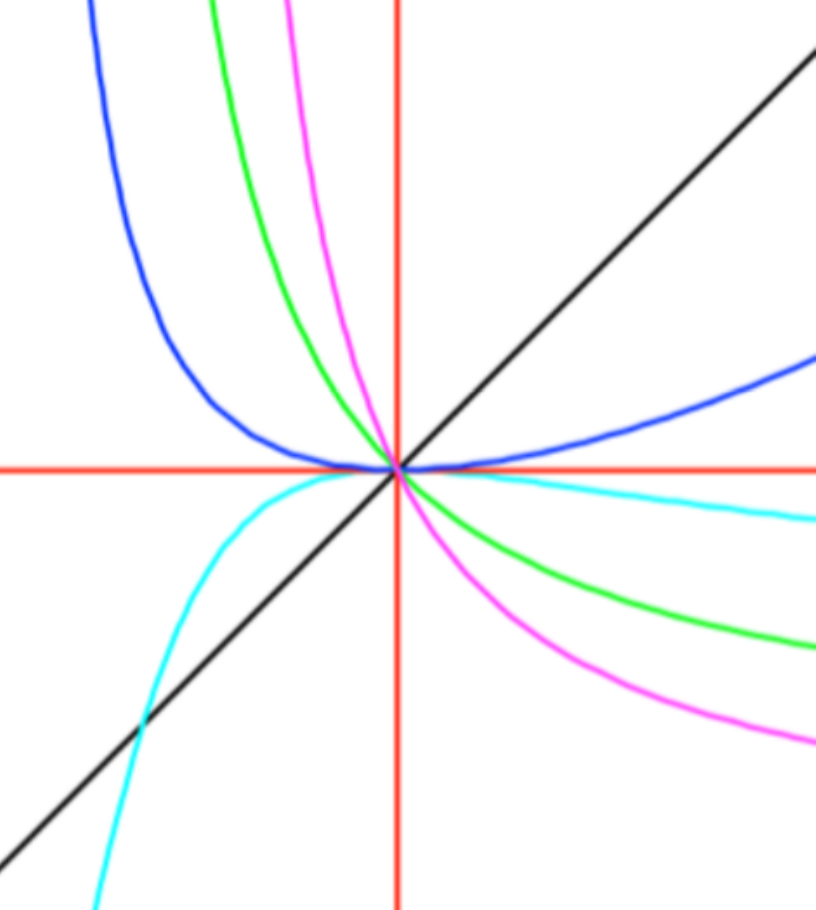} 
   \caption{Picture of $\sL \cup \sZ$ in $\P^1 \times \P^1$ near $(\infty,\infty)$.  The line $x = y$ is drawn in black, and the lines $x = \infty$ and
   $y = \infty$ are drawn in red.}
   \label{ClosedW-fig}
\end{figure}

\subsection{The quotient of augmented deformation space}

Let $\AugModuli_f$ be the quotient of $\AugTeich_f$ by the action of $L_f$,  and let $\AugV_{f,\iota}$ be the image of $\AugDef_{f,\iota}$
in $\AugModuli_f$.  Our goal in this section is to concretely describe a subspace $X \subset \AugV_{f,\iota}$ satisfying the properties in Proposition~\ref{suff-prop}.

First we recall that the elements of $\AugV_{f,\iota}$ are the elements of $\AugModuli_f$ that equalize the two maps
$$
p_L, p_U : \AugModuli_f \rightarrow \AugModuli_{(S^2,A)}.
$$
Each stratum of $\AugModuli_f$ is described 
by a partition of $\{0,1,\infty,y,z\}$ into two or three sets by an admissible multi-curve $\gamma$, as in Figure~\ref{partition-fig}.
Here, the empty multi-curve corresponds to the principal stratum $\Moduli_f \subset \AugModuli_f$.
\begin{figure}[htbp] 
   \centering
   \includegraphics[width=3.5in]{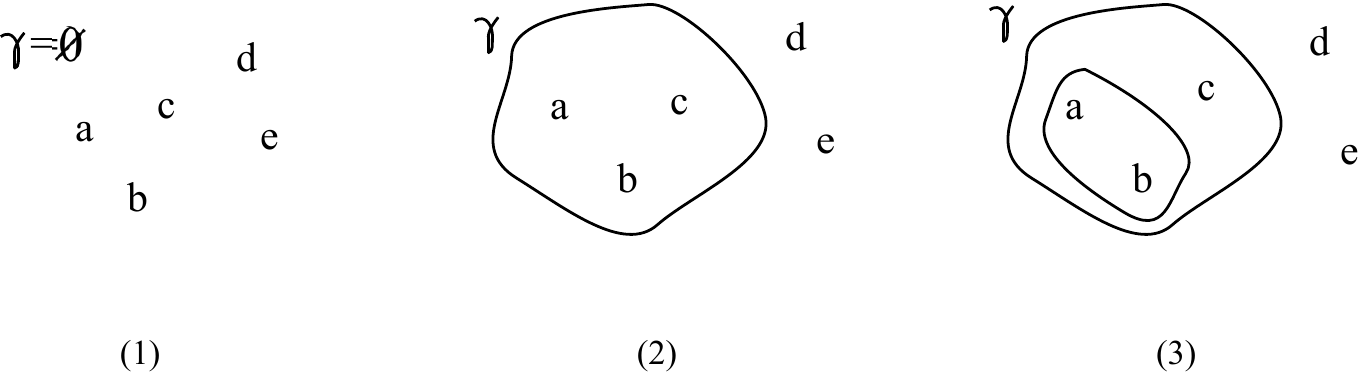} 
   \caption{Possible partitions of five points by an admissible multi-curve.}
   \label{partition-fig}
\end{figure}
The corresponding stable curves are shown in Figure \ref{stable5-fig}, with each $\P^1$, homeomorphic to $S^2$, is drawn as a line.
\begin{figure}[htbp] 
   \centering
   \includegraphics[width=4in]{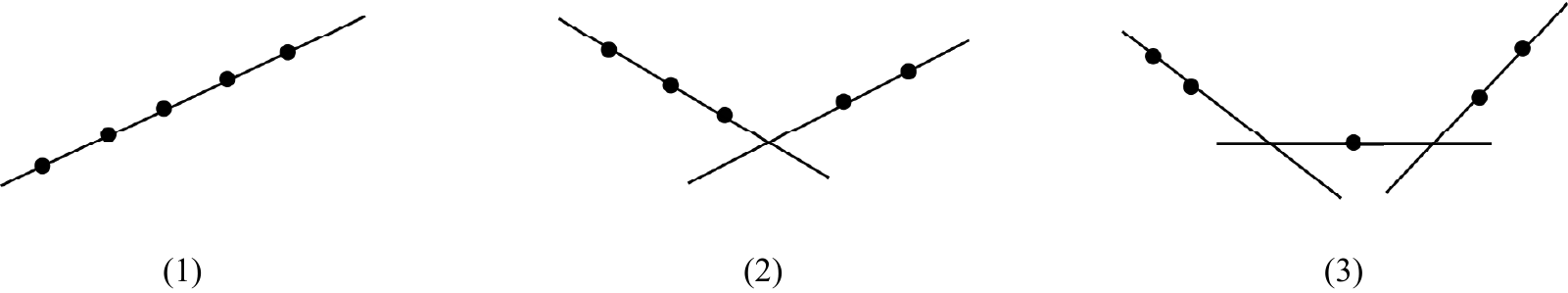} 
   \caption{The three types of stable curves for $(S^2,B)$ where $B$ has five elements.
   The two left define one-dimensional strata, and the right defines a point stratum.}
   \label{stable5-fig}
\end{figure}

For $\AugModuli_{(S^2,A)}$ there are only two isomorphism types (shown in Figure~\ref{stable4-fig}).  The left picture depicts points 
belonging
to the main component
$\Moduli_{(S^2,A)}\subset\AugModuli_{(S^2,A)}$,  which is isomorphic to a thrice punctured sphere, while 
the right picture depicts one of the three single point boundary points.

\begin{figure}[htbp] 
   \centering
   \includegraphics[width=2.5in]{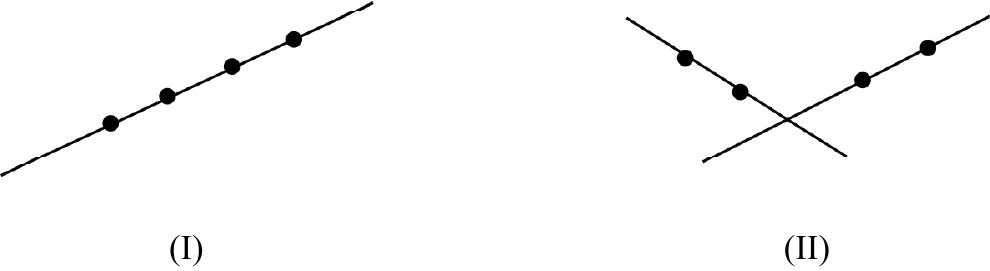} 
   \caption{The two topological homeomorphism types of stable curves for $(S^2,A)$ where $A$ has four elements.
}
   \label{stable4-fig}
\end{figure}

The elements $\alpha \in \AugModuli_f$ that lie in positive dimensional strata must be of the form (1) or (2)
in Figure~\ref{stable5-fig}.  Those of type (1) lie in $\Moduli_f = \P^1 \times \P^1 \setminus \sL \cup \sZ$
and map under both $p_U$ and $p_L$  to elements of $\AugModuli_{(S^2,A)}$ of type (I).
Those of type (2) divide into four subtypes: those that correspond to a partition of the form
\begin{enumerate}
\item [(2a)] $\{a,b,z\}\cup\{c,\infty\}$; 
\item [(2b)] $\{a,\infty,z\}\cup\{b,c\}$;
\item [(2c)] $\{a,b,c\} \cup \{\infty,z\}$; or
\item [(2d)] $\{a,b,\infty\}\cup\{c,z\}$.
\end{enumerate}

For types (2a) and (2b), $p_L$ maps $\alpha$ to an element in $\AugModuli_{(S^2,A)}$ of type (II),
while for types (2c) and (2d), $p_L$ maps $\alpha$ to one of type (I).
For types (2a) and (2d) $p_U$ maps $\alpha$ to an element of type (II), while for type
(2b) and (2c) $\alpha$ could apriori map to an element of type (I) or (II).  This is because the critical
values $\infty$ and $z$ lie in the same component.  Thus the rational map $F : \sC_U \rightarrow \sC_L$ must have two isomorphic
irreducible components in $\sC_U$ lying over the unramified component in $\sC_L$ upon which the distinguished points will be distributed.

From this we can reduce the types that can be in $\AugV$ to (2a), (2d), as well as 
 (2b) and (2c) under the condition that
 the  preimage under $F$ of the distinguished points in the unramified component lie on the same component of $\sC_U$.
 In the allowable case of types (2a) and (2b), we see that for $p_L(\alpha)$ and $p_U(\alpha)$ to be equal the images of the two maps
 in $\AugModuli_{(S^2,A)}$ must give the partition
 $$
 \{0,1\} \cup \{x,\infty\}.
 $$
 Thus, the partition given in (2a) can only be
  $$
 \{0,1,z\} \cup \{y,\infty\}
 $$
 and for (2b) it can only be
 $$
 \{\infty, y, z\} \cup \{0,1\}.
 $$
 For type (2c), $p_L(\alpha)$ and $p_U(\alpha)$ must be equal, and $F$ defines the isomorphism on stable curves.
Taking into account the combinatorics of $f$, this implies equality of the cross ratios:
 $(0,\infty;1,x)$ and $(\infty,1; x,0)$, which is false under the assumption that $x \notin \{0,1,\infty\}$. 
 
 Let $\sA_1, \sA_2 \subset \AugModuli_f$ be the subsets corresponding to the partitions
  $$
\sA_1:  \{\infty, y,z\} \cup \{0,1\},
 $$
  and 
 $$
\sA_2:  \{0,1,z\} \cup \{y,\infty\}.
 $$
 Then each of these is isomorphic to $\Mod_{0,4} \times \Mod_{(0,3)}$.
 Furthermore, the closures of $\sA_1$ and $\sA_2$ in $\AugModuli_f$ intersect at the point corresponding to the partition
 $$
 \{0,1\} \cup \{z\} \cup \{y,\infty\}.
 $$

We have shown the following.

\begin{proposition} \label{connectedness-prop}
The pure 1-dimensional algebraic set $\V \cup \sA_1 \cup \sA_2 \subset \AugModuli_f$ is contained in $\AugV_{f,\iota}$, and its
complement is a finite set of points (possibly empty).
\end{proposition}
 
 \begin{remark} We leave the question of whether $\AugV_{f,\iota}$ is connected (in this case, and in general) to future study. \end{remark}

\subsection{Blowups}\label{blowup-sec}

In this section, we find a connected pure 1-dimensional algebraic subset $X \subset \AugV$
that contains $\V$, and whose complement in $\AugV$ is finite.

Let
$$
\sigma : \widetilde {\P^1 \times \P^1} \rightarrow \P^1 \times \P^1
$$
be sequence of blowups defined as follows.
First blowup the points $(0,0)$, $(1,1)$ and $(\infty,\infty)$ to get the exceptional curves $E_0,E_1,E_\infty$.
Next blowup the point of intersection $q \in E_\infty \cap \widehat L_y$, where
$\widehat L_y$ is the proper transform of $\{y=\infty\}$. Let $E_q$ be the
exceptional divisor.   The union of curves is drawn in Figure~\ref{blowup2-fig} (compare Figure~\ref{affine-fig}).

\begin{figure}[htbp] 
   \centering
   \includegraphics[width=2.5in]{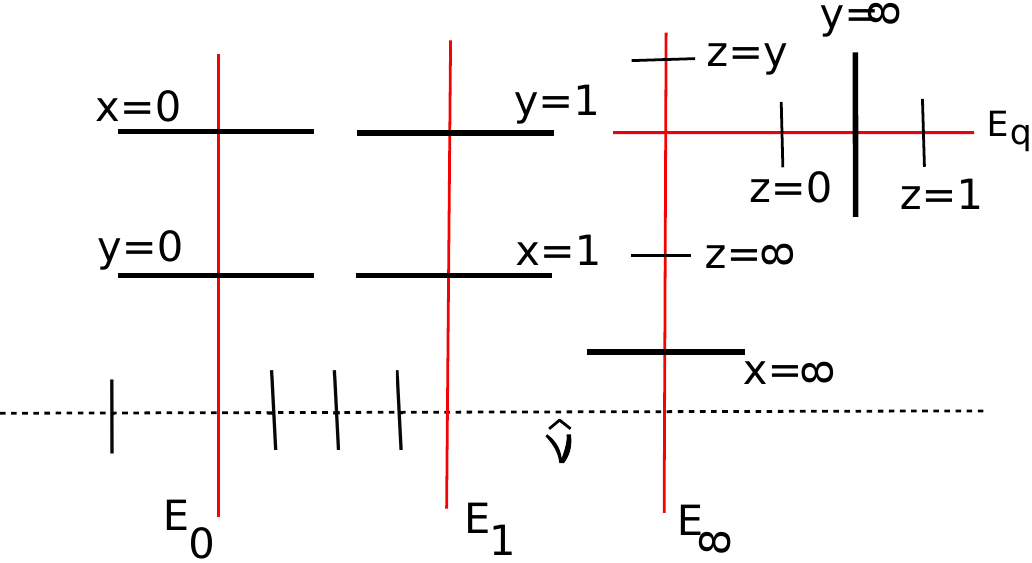} 
   \caption{The proper transform of $\sV$ is drawn as a dotted line, and the exceptional curves $E_0,E_1,E_\infty$ and $E_q$
  are drawn in red. Intersections with $\sL$ are indicated as thickened black line, and intersections with $\sZ$ are indicated with
   a thin black line}
   \label{blowup2-fig}
\end{figure}

\begin{lemma}  The map $\sigma$ has the following properties:
\begin{enumerate}
\item $\sigma$ restricts to an isomorphism on $\widetilde {\P^1 \times \P^1} \setminus \sigma^{-1}(Q)$; and
\item the total transform $\widetilde \V$ and the proper transform $\widehat {\sL \cup \sZ}$ meet in normal crossing
singularities.
\item inclusion induces a surjection on fundamental groups
$$
\pi_1(\sN(\widetilde \V) \cap \Moduli_f) \rightarrow \pi_1(\Moduli_f).
$$
\end{enumerate}
\end{lemma}

\begin{proof}  Properties (1) and (2) follow from the definitions, and property (3) follows from Lemma~\ref{surj-lem}.
\end{proof}

Let 
$$
X = (\widehat \V \cup E_\infty \cup E_q) \setminus \widehat {\sL \cup \sZ}.
$$
Then $X$ is connected since  the punctures of $\widetilde \V$ at intersections with $\widehat {\sL \cup \sZ}$ occur
only in smooth points of $\widetilde \V$.  Since $\sigma$ is an isomorphism over $\V$, there is a natural inclusion
$\nu: \V \hookrightarrow X$.

\begin{lemma} The inclusion $\V \hookrightarrow \AugV$ induced by $\Def_{f,\iota} \hookrightarrow \AugDef_{f,\iota}$ factors  as
$\xi \circ \nu$ for some embedding 
$$
\xi: X \hookrightarrow \AugV.
$$
\end{lemma}

\begin{proof}   
Recall from Proposition~\ref{connectedness-prop} the subset $\V \cup \sA_1 \cup \sA_2 \subset \AugV$.
In what follows we define embeddings
\begin{eqnarray*}
\kappa_\infty:  E_\infty\setminus \widehat {\sL \cup \sZ} &\rightarrow& \sA_1\\
\kappa_q :  E_q \setminus \widehat {\sL \cup \sZ} &\rightarrow&\sA_2
\end{eqnarray*}
that together with the projection $\widehat V \rightarrow \V$ 
extend to define $\xi$.

When $\alpha \in \Moduli_f$ approaches a general point of $\sL \cup \sZ$, it corresponds to 
two points in $\{0,1,\infty,y,z\}$ coming together.  For the lines $(y=0$, $y=1$ and $y=\infty$, the pairs
$\{y,0\}$, $\{y,1\}$ and $\{y,\infty\}$ approach each other,
while simultaneously $\{x,1\}$, $\{1,\infty\}$ and $\{1,0\}$ approach each other.
Near the lines $x = 0$, $x=1$, $x=\infty$
the pairs $\{x,0\}$) (and $\{0,\infty\}$), $\{x,1\}$ (and $\{0,y\}$) and $\{x,\infty\}$ (and $\{0,1\}$) approach each other.
For $z$ approaching $0,1,\infty$ or $y$, we have
\begin{eqnarray*}
\{z,0\}  &\mbox{near}& 1 - 2x + x^2 - y = 0; x \neq 0, 1; x + y \neq 1\\
\{z,1\}  &\mbox{near}& x^2 + y = 1; x \neq 0; x+y \neq 1\\
\{z,\infty\} &\mbox{near}& x = 0, x=1, \mbox{or} \  x+y=1\\
\{z,y\}  &\mbox{near}& 2xy + x^2 - y - 2x =0
\end{eqnarray*}

We choose local coordinates for a neighborhood of $E_\infty$ as follows. Let $\overline x = 1/x$ and $\overline y = 1/y$ be coordinates for an
open neighborhood of $p_\infty = (\infty,\infty)$ in $\P^1 \times \P^1$, so that in the coordinates $(\overline x,\overline y)$ we
have $p_\infty = (0,0)$.  Let $u$ (and $\overline u = \frac{1}{u}$) be coordinates for $E_\infty$.  Then  a neighborhood of $E_\infty$
in $\widetilde {\P^1 \times \P^1}$ is isomorphic to the algebraic subset of $\C^2 \times \P^1$ defined by
$$
\{(\overline x,\overline y) \times u \in \C^2 \times \P^1\ | \overline x =   u \overline y\},
$$
and
$$
E_\infty = \{(\overline x, \overline y) \times u \in \C^2 \times \P^1\ | \ \overline x = \overline y = 0\}.
$$
With respect to the $\overline x, \overline y$ coordinates, we have
\begin{eqnarray*}
z &=& - \frac{\overline x^4 \overline y^2 (1 - 2x + x^2 - y)^2}{4 \overline x^4\overline y^2x (x-1)(x+y-1)}\\
&=& - \frac{(\overline x^2\overline y - 2\overline x\overline y + \overline y - \overline x^2)^2}
{4 \overline x\overline y(1 - \overline x)(\overline y + \overline x -\overline x \overline y)}\\
\end{eqnarray*}

Using the identity $\overline x = u \overline y$, we have
\begin{eqnarray*} 
z &=&- \frac{(u^2 \overline y^2 - 2 u\overline y + 1 - u^2 \overline y)^2}
{4 u \overline y(1-u\overline y)(1 + u (1 - \overline y))}.
\end{eqnarray*}
To each $(u, \overline y)$, associate the point in $\AugModuli_{(S^2,B)}$
defined by
$$
([u,\infty ; y, z] , *)
$$
where the first component is the cross ratio of the points $u,\infty, y,z$ and the second
component is the unique triple of points in $\P^1$ up to automorphism.   This defines a point in 
$\AugModuli_{(S^2,B)}$.  

Since cross ratio is preserved under automorphisms of $\P^1$ we have
$$
[u,\infty: y, z] = [u\overline y,\infty, 1,z \overline y].
$$
As $\overline y$ approaches $0$, the cross ratio approaches
$$
[0,\infty; 1, -\frac{1}{4 u(1+u)}] 
$$
and is degenerate only when  $u = 0$ ($x = \infty$),  $u = -1$ ($z = \infty$), $u = -\frac{1}{2}$ ($z = y$), or $u = \infty$ ($y = \infty$).

We thus have a well-defined embedding:
\begin{eqnarray*}
\kappa_\infty : E_\infty \setminus \widehat {\{\sL \cup \sZ\}} &\rightarrow& \sA_1\\
u &\mapsto& ([u, \infty; y,z], *).
\end{eqnarray*}
Here $u$ corresponds to a line through $(\infty,\infty)$ in $\P^1 \times \P^1$, and we can think of the contracting curve $\gamma$
as a small loop on this line around $(\infty,\infty)$.

The intersection of $E_\infty$ with $E_q$ occurs at the point on $E_\infty$ corresponding to $(\overline y,\overline u) = (0,0)$, 
and $E_q$ has
a neighborhood parameterized by $((\overline y, \overline u),v) \in \C^2 \times \P^1$ where $\overline y = v \overline u$.  
Then
$\overline x = u \overline y = v$ and
\begin{eqnarray*}
z
&=& -\frac {(1 -2x + x^2 - y)^2}{4 x(x-1)(x+y-1)}\\
&=& 
 - \frac{(\overline u(1 -2\overline v + \overline v^2) - \overline v)^2}{4 \overline v (\overline v-1)(\overline v + \overline u \overline v - 1)}\\
\end{eqnarray*}
and as $\overline u$ goes to $0$, 
$$
z = -\frac{\overline v}{4 (\overline v-1)^2} = -\frac{v}{4(1-v)^2}.
$$
Then we have
$$
[v,0;1,z] =  [1,0; \overline v, \overline v z] = [1,0; \overline v, -\frac{1}{4(1-v)^2}].
$$
The cross ratio is degenerate when $v = \infty$ (where
$\widehat {\{y=\infty\}}$ meets $E_q$), $ v = 1$ (where $E_\infty$ and $E_q$ intersect), 
and where $z = 0$ and $z = 1$ (corresponding to two distinct values of $v$ not equal to $1$ or $\infty$).
The points correspond to $v \in E_q \setminus \widehat {\sL \cup \sZ}$ except at the point
where $v = 1$.  

For this point, consider the 2-component multicurve $\gamma$ on $S^2 \setminus B$
determined by two loops on the planes defined by
$E_\infty$ and $\widehat{\{y=\infty\}}$ around $(\overline y, \overline u) = (0,0)$.  
The corresponding point $\alpha_0 \in \AugModuli_f$ corresponds to the partition
$$
\{y,\infty\} \cup \{z\} \cup \{0,1\}.
$$
Define
\begin{eqnarray*}
\kappa_q : E_q \setminus \widehat {\{\sL \cup \sZ\}} &\rightarrow& \sA_2\\
v &\mapsto& \left \{
\begin{array}{ll}
 ([0,1; \infty, \overline v], *) &\mbox{if $v \neq 1$}\\
 \alpha_0 &\mbox{if $v = 1$}
 \end{array}
 \right .
\end{eqnarray*}

Then $\kappa_\infty$ and $\kappa_q$ extend to a morphism $\xi : X \rightarrow \AugV$.
\end{proof}

\begin{proof}[Proof of Theorem~\ref{main-thm}]  Fix $d_0 \in \Def_{f,\iota}$, and let $\sD_0$ be the connected component of 
$\Def_{f,\iota}$ that contains $d_0$.  Let $\overline \sD_0$ be the closure of $\sD_0$ in $\AugDef_{f,\iota}$.  Then 
$\overline \sD_0$ maps to $\widehat \V$ under the projection from $\AugTeich_f \rightarrow \AugModuli_f$.
Let $\widehat \sD_0 \subset \AugTeich_f$  be the connected component of the preimage of $\widehat \V$ that intersects $\sD_0$.
Then, since $\widehat \V$ is closed in $X$ and 
$X$ is connected, we can find a connected component $Y_0$ of the preimage of $X$ in $\AugTeich_f$ that contains $\widehat D_0$,
and 
$\widehat \sD_0$ is the closure
in $Y_0$ of $\sD_0$.  That is, $\widehat \sD_0 = \overline \sD_0$.

We claim that there is an element 
$$
g \in S_{f,\iota} \cap \mbox{Image}( \pi_1(N(X) \cap \Moduli_f) \rightarrow \pi_1(\Moduli_f))
\setminus \mbox{Image}(\pi_1(\widehat V) \rightarrow \pi_1(\Moduli_f)).
$$
Let $v_0$ be the image of $d_0$ in $\V$, and let $\gamma$ be a close path starting at $v_0$ and passing along $X$ to 
a point near an intersection of $\widehat \sZ$ with $E_\infty \cup E_q$, forming a small loop around that intersection point,
and returning along the original path back to $v_0$.  Then $\gamma$ is not homotopic in $\AugModuli_f$ to an closed 
path on $\widehat V$ since $\widehat Z$ and $\widehat V$ do not intersect.
 Let $g$ be the image of $\gamma$ in $\pi_1(\Moduli_f)$ after pushing off the boundary of $\AugModuli_f$
into $\Moduli_f$.

Since $g \notin \mbox{Image}(\pi_1(\widehat V) \rightarrow \pi_1(\Moduli_f))$,  $g$ does not preserve
$\widehat D_0$, and hence $g(\widehat \sD_0)$ and $\widehat \sD_0$ are disjoint.  These are the closures of
$\sD_0$ and $g(\sD_0)$ in $\AugDef_{f,\iota}$ and the claim follows.
\end{proof}

\subsection{Connectivity of augmented deformation space}\label{connectedness-sec}

We finish this paper by giving a sufficient condition for $\AugDef_{f,iota}$ connected in the $\Per_4(0)^*$ case.

By Corollary~\ref{connectedness-cor}, we know that $\widetilde V$ has a  regular neighborhood $N$ in $\widetilde {\P^1 \times \P^1}$ so
that 
$$
\pi_1(N \cap \Moduli_f) \rightarrow \pi_1(\Moduli_f)
$$
is surjective. 

As a closure of $\Moduli_f$, $\widetilde {\P^1 \times \P^1}$ is birationally equivalent to $\AugModuli_f$ (but not isomorphic). By the
birational theory of 
complex projective surfaces, there is a minimal smooth surface $Z$ with birational morphisms to $\AugModuli_f$ and $\widetilde {\P^1 \times \P^1}$ 
$$
\xymatrix{
&Z\ar[dl]\ar[dr]\\
\AugModuli_f &&\widetilde {\P^1 \times \P^1}.
}
$$
Lifting $U$ to $Z$ and projecting to $\AugModuli_f$ gives a regular neighborhood $U'$ of $\overline  \V \cup \mathcal K$, where $\mathcal K$ is a set
of boundary curves in $\AugModuli_f$.   Since $\Moduli_f$ is a smooth subset of both $\AugModuli_f$ and $\widetilde {\P^1 \times \P^1}$ the
two projects of $Z$ are isomorphisms over $\Moduli_f$.  Thus
$$
U' \cap \Moduli_f = U \cap \Moduli_f,
$$
 and hence
$$
\pi_1(U' \cap \Moduli_f) \rightarrow \pi_1(\Moduli_f)
$$
is surjective.  It follows that $U'$ has a connected preimage in $\AugTeich_f$ an that the preimage of $\overline V \cup \mathcal K$ in $\AugTeich_f$
is connected.

We have thus shown the following sufficient condition for $\AugDef_{f,\iota}$ to be connected.

\begin{proposition} If $\AugV$ is connected and is Zariski dense in $\overline \V \cup \mathcal K$, then $\AugDef_{f,\iota}$ is connected.
\end{proposition}

\bibliographystyle{plain}
\bibliography{../../math}

\end{document}